\documentclass{amsart}
\usepackage{amssymb}
\usepackage[pdfencoding=auto]{hyperref}
\usepackage{mathtools}
\usepackage{comment}
\usepackage{enumitem}
\usepackage{graphicx}
\usepackage{ifthen}
\usepackage[svgnames]{xcolor}
\usepackage{scalerel}
\usepackage{makebox}
\usepackage{mathtools}

\usepackage[capitalise]{cleveref}

\setlist{}
\setenumerate{leftmargin=*,labelindent=0\parindent}
\setitemize{leftmargin=\parindent}

\usepackage{tikz}
\usepackage{tikz-cd}

\usetikzlibrary{decorations.markings}
\usetikzlibrary{decorations.pathmorphing}
\tikzset{
  module/.style={
    postaction={decorate},
    decoration={
      markings,
      mark=at position #1 with {\arrow{|}}}},
  module/.default=0.5,
  we/.style=
  { postaction={%
      decorate,
      decoration={
        markings,
        mark=at position #1 with {%
          \node[transform shape, yshift=.2em]{%
            \resizebox{0.5em}{!}{$\sim$}};}}}},
  we/.default=0.5,
  we'/.style=
  { postaction={%
      decorate,
      decoration={
        markings,
        mark=at position #1 with {%
          \node[transform shape, yshift=-.2em, rotate=180]{%
            \resizebox{0.5em}{!}{$\sim$}};}}}},
  we'/.default=0.5,
  iso/.style=
  { postaction={%
      decorate,
      decoration={
        markings,
        mark=at position #1 with {%
          \node[transform shape, yshift=.4em]{%
            \resizebox{0.5em}{!}{$\cong$}};}}}},
  iso/.default=0.5,
  iso'/.style=
  { postaction={
      decorate,
      decoration={
        markings,
        mark=at position #1 with {%
          \node[transform shape, yshift=-.4em, rotate=180]{%
            \resizebox{0.5em}{!}{$\cong$}};}}}},
  iso'/.default=0.5,
}

\tikzset{
  wearrow/.style={->, we},
  wedashed/.style={->, dashed, we},
  wedotted/.style={->, dotted, we},
  tfibarrow/.style={->>, we=0.45},
  tfibdotted/.style={->>,dotted, we=0.45},
  tfibdashed/.style={->>,dashed, we=0.45},
  tcofarrow/.style={>->, we},
  tcofdashed/.style={>->, dashed, we},
  uwearrow/.style={->, we'},
  uwedashed/.style={->, dashed, we'},
  uwedotted/.style={->, dotted, we'},
  utfibarrow/.style={->>, we'=0.45},
  utfibdotted/.style={->>,dotted, we'=0.45},
  utfibdashed/.style={->>,dashed, we'=0.45},
  utcofarrow/.style={>->, we'},
  isoarrow/.style={->, iso},
  isodashed/.style={->, dashed, iso},
  uisoarrow/.style={->, iso'},
  uisodashed/.style={->, dashed, iso'},
  isocell/.style={=>, iso},
  isocelldashed/.style={=>, dashed, iso},
  uisocell/.style={=>, iso'},
  uisocelldashed/.style={=>, dashed, iso'}
}

\newtheorem{thm}{Theorem}[section]
\newtheorem{lem}[thm]{Lemma}
\newtheorem{prop}[thm]{Proposition}

\newtheorem{cor}[thm]{Corollary}

\theoremstyle{definition}
\newtheorem{defn}[thm]{Definition}
\newtheorem{ex}[thm]{Example}

\crefname{lem}{Lemma}{Lemmas}
\crefname{thm}{Theorem}{Theorems}
\crefname{defn}{Definition}{Definitions}
\crefname{notn}{Notation}{Notations}
\crefname{const}{Construction}{Constructions}
\crefname{prop}{Proposition}{Propositions}
\crefname{rmk}{Remark}{Remarks}
\crefname{cor}{Corollary}{Corollaries}
\crefname{ex}{Example}{Examples}
\crefname{nex}{Non-Example}{Non-Examples}

 \makeatletter
 \let\c@equation\c@thm
 \makeatother
 \numberwithin{equation}{section}

\newcommand{\op}{\textup{op}}
\newcommand{\id}{\textup{id}}
\newcommand{\ob}{\textup{ob}}

\makeatletter
\def\@mathlower#1#2#3{\setbox0=\hbox{$\m@th#2#3$}\lower#1\ht0\box0}
\def\mathlower#1#2{\mathpalette{\@mathlower{#1}}{#2}}
\makeatother

\newcommand{\slicel}[2]{\vphantom{#2}^{{#1}/}\mkern-2mu{#2}}

\newcommand{\DDelta}{\mathbf{\Delta}}

\newcommand{\cat}[1]{\textup{\textsf{#1}}}

\newcommand{\Cat}{\cat{Cat}}
\newcommand{\sSet}{\cat{sSet}}
\newcommand{\sCat}{\cat{sCat}}

\DeclareSymbolFont{bbold}{U}{bbold}{m}{n}
\DeclareSymbolFontAlphabet{\mathbbb}{bbold}
\newcommand{\catone}{\ensuremath{\mathbbb{1}}}
\newcommand{\cattwo}{\ensuremath{\mathbbb{2}}}

\newcommand{\cA}{\mathcal{A}}
\newcommand{\cB}{\mathcal{B}}
\newcommand{\cC}{\mathcal{C}}
\newcommand{\cD}{\mathcal{D}}

\newcommand{\cU}{\mathcal{U}}
\newcommand{\cV}{\mathcal{V}}
\newcommand{\cW}{\mathcal{W}}
\newcommand{\cY}{\mathcal{Y}}

\newcommand{\cstar}{\operatorname{disc}}
\newcommand{\hcnerve}{\mathfrak{N}}

\newcommand{\aamalg}[1]{\underset{#1}{{\amalg}}}

\begin{document}

\title{Pushouts of Dwyer maps are \texorpdfstring{$(\infty,1)$}{(infinity,1)}-categorical}

\author{Philip Hackney}
\address{Department of Mathematics,
University of Louisiana at Lafayette, LA, USA
}
\email{philip@phck.net} 

\author{Viktoriya Ozornova}
\address{Max Planck Institute for Mathematics, Bonn, Germany}
\email{viktoriya.ozornova@mpim-bonn.mpg.de}

\author{Emily Riehl}
\address{Department of Mathematics\\Johns Hopkins University \\ 
Baltimore, MD 
 \\ USA}
\email{eriehl@jhu.edu}

\author{Martina Rovelli}
\address{Department of Mathematics and Statistics, 
University of Massachusetts at 
Amherst,
MA, 
USA
}
\email{rovelli@math.umass.edu}

\date{\today}

\begin{abstract}
The inclusion of 1-categories into $(\infty,1)$-categories fails to preserve colimits in general, and pushouts in particular. In this note, we observe that if one functor in a span of categories belongs to a certain previously-identified class of functors, then the 1-categorical pushout is preserved under this inclusion.
Dwyer maps, a kind of neighborhood deformation retract of categories, were used by Thomason in the construction of his model structure on 1-categories.
Thomason previously observed that the nerves of such pushouts have the correct weak homotopy type.
We refine this result and show that the weak homotopical equivalence is a weak categorical equivalence. We also identify a more general class of functors along which 1-categorical pushouts are $(\infty,1)$-categorical.
\end{abstract}

\thanks{This material is based upon work supported by the National Science Foundation under Grant No.\ DMS-1440140, which began while the authors were in residence at the Mathematical Sciences Research Institute in Berkeley, California, during the Spring 2020 semester. The authors learned a lot from fruitful discussions with the other members of the MSRI-based working group on $(\infty,2)$-categories. This work was supported by a grant from the Simons Foundation (\#850849, PH). The second author thankfully acknowledges the financial support by the DFG grant OZ 91/2-1 with the project nr.~442418934. The third author is also grateful for support from the NSF via DMS-1652600 and DMS-2204304, from the ARO under MURI Grant W911NF-20-1-0082, and by the Johns Hopkins President's Frontier Award program. The fourth author is deeply appreciative of the Mathematical Sciences Institute at the Australian National University for their support during the pandemic year and is grateful for support from the NSF via DMS-2203915. The authors also wish to thank the anonymous referee for \cite{HORR} who pointed out that their originally-claimed proof of this result, which originally appeared there, sufficed only to cover the case used there. The anonymous referee for this paper suggested a strategy to greatly simplify our original proof of the general result, which allowed us to also prove more general theorems.}

\maketitle

\setcounter{tocdepth}{1}
\tableofcontents

\section{Introduction}

Classical 1-categories define an important special case of $(\infty,1)$-categories. The fact that $(\infty,1)$-category theory restricts to ordinary 1-category can be understood, in part, by the observation that the inclusion of 1-categories into $(\infty,1)$-categories is full as an inclusion of $(\infty,2)$-categories. This full inclusion is reflective---with the left adjoint given by the functor that sends an $(\infty,1)$-category to its quotient ``homotopy category''---but not coreflective and as a consequence colimits of ordinary 1-categories need not be preserved by the passage to $(\infty,1)$-categories. Indeed there are known examples of colimits of 1-categories that generate non-trivial higher-dimensional structure when the colimit is formed in the category of $(\infty,1)$-categories. 

For example, consider the span of posets:
\begin{equation}\label{eq:poset-non-example}
\begin{tikzcd}[sep=tiny]
 & & [-6pt] \bullet \arrow[dd] \arrow[dl] \arrow[dr] \\ [+5pt] & \bullet \arrow[dr, shorten >= -.25em, shorten <= -.45em] &  & [-6pt] \bullet \arrow[dl, shorten >= -.25em, shorten <= -.45em] \\ [-3pt] & & \bullet \\ \bullet \arrow[uur] \arrow[urr] \arrow[rr] & & \bullet \arrow[u, shorten >= -.25em, shorten <= -.25em] & & \bullet \arrow[ull] \arrow[uul] \arrow[ll]
\end{tikzcd} \hookleftarrow
\begin{tikzcd}[sep=tiny]
 & & [-7pt] \bullet  \arrow[dl] \arrow[dr] \\ [+5pt] & \bullet &  & [-7pt] \bullet  \\ [-3pt] & & \makebox*{$\bullet$}{$~$} \\ \bullet \arrow[uur] \arrow[rr] & & \bullet  & & \bullet  \arrow[uul] \arrow[ll]
\end{tikzcd} \to \quad \bullet
\end{equation}
The pushout in 1-categories is the arrow category $\bullet \to \bullet$, while the pushout in $(\infty,1)$-categories defines an $(\infty,1)$-category which has the homotopy type of the 2-sphere.

As a second example, let $M$ be the monoid with five elements $e, x_{11}, x_{12}, x_{21}, x_{22}$ and multiplication rule given by $x_{ij}x_{k\ell}=x_{i\ell}$.
Inverting all elements of $M$ yields the trivial group.
That is, if one considers $M$ as a 1-category with a single object, then the pushout of the span $M \leftarrow \amalg_{M} \cattwo \rightarrow \amalg_{M} \mathbb{I}$ (where $\cattwo$ is the free-living arrow and $\mathbb{I}$ is the free-living isomorphism) in categories is the terminal category $\catone$. 
On the other hand, the pushout of this span in $(\infty,1)$-categories is the $\infty$-groupoid $S^2$ as follows from  \cite[Lemma]{FiedorowiczCounterexample}. 
The results of \cite{McDuffMonoids} imply that this example is generalizable to a vast class of monoids.

More generally, the Gabriel--Zisman category of fractions $\cC[\cW^{-1}]$ is formed by freely inverting the morphisms in a class of arrows $\cW$ in a 1-category $\cC$. 
This can also be constructed as a pushout of 1-categories of the span
\begin{equation*}
\label{spanLoc}
 \begin{tikzcd} \cC & \aamalg{w \in \cW} \cattwo \arrow[l] \arrow[r, hook] & \aamalg{w \in \cW} \mathbb{I}
\end{tikzcd}   
\end{equation*}
where each arrow in $\cW$ is replaced by a free-living isomorphism. By contrast, the $(\infty,1)$-category defined by this pushout is modelled by the Dwyer--Kan simplicial localization, which has non-trivial higher dimensional structure in many instances \cite{DwyerKan:SLC}, \cite[Lemma 18]{Stevenson:CMSSL}, \cite[p.~168]{JoyalVolumeII}. Indeed, all $(\infty,1)$-categories arise in this way \cite{BarwickKan:RCAMHTHT}.

As the examples above show, pushouts of 1-categories in particular are problematic. Our aim in this paper is to prove that a certain class of pushout diagrams of 1-categories are guaranteed to be $(\infty,1)$-categorical. The requirement is that one of the two maps in the span that generates the pushout belong to a class of functors between 1-categories first considered by Thomason under the name ``Dwyer maps''  \cite[Definition 4.1]{ThomasonModelCat} that feature in a central way in the construction of the Thomason model structure on categories.

\begin{defn}[Thomason]\label{defn:Dwyer-map}
A full sub-$1$-category inclusion $I \colon \cA \hookrightarrow \cB$ is \textbf{Dwyer map}
if the following conditions hold.
\begin{enumerate}[label=(\roman*)]
\item The category $\cA$ is a \emph{sieve} in $\cB$, meaning there is a necessarily unique functor $\chi \colon \cB \to \cattwo$ with $\chi^{-1}(0) = \cA$. We write $\cV:=\chi^{-1}(1)$ for the complementary \emph{cosieve} of $\cA$ in $\cB$. 
\item The inclusion $I \colon \cA \hookrightarrow\cW$ into the \emph{minimal cosieve}\footnote{Explicitly $\cW$ is the full subcategory of $\cB$ containing every object that arises as the codomain of an arrow with domain in $\cA$.} $\cW \subset \cB$ containing $\cA$ admits a right adjoint left inverse $R\colon \cW\to \cA$, a right adjoint for which the unit is an identity.
\end{enumerate}
\end{defn}

Schwede describes Dwyer maps as ``categorical analogs of the inclusion of a neighborhood deformation retract'' \cite{SchwedeOrbispaces}. In fact, many examples of Dwyer maps are more like deformation retracts, in that the cosieve $\cW$ generated by $\cA$ is the full codomain category $\cB$.

\begin{ex}\label{DwyerExamples}$\quad$
\begin{enumerate}[label=(\roman*), ref=\roman*]
\item \label{basicDwyer}
The vertex inclusion $0 \colon \catone \to \cattwo$ is a Dwyer map, with $! \colon \cattwo \to \catone$ the right adjoint left inverse.
The other vertex inclusion $1 \colon \catone \to \cattwo$ is not a Dwyer map.
\item \label{ex:new-terminal-Dwyer}  Generalizing the previous example, if $\cA$ is a category with a terminal object and  $\cA^{\triangleright}$ is the category which formally adds a new terminal object, then the inclusion $\cA \hookrightarrow \cA^{\triangleright}$ is a  Dwyer map.\footnote{If $\cA$ does not have a terminal object, then $\cA \to \cA^{\triangleright}$ need not be a Dwyer map. Indeed, if $\cA=\catone\amalg\catone$, the only cosieve containing $\cA$ is $\cA^\triangleright$ itself, and there cannot be a right adjoint $\cA^\triangleright \to \cA$ as $\cA$ does not have a terminal object. But see \cref{discretely flat examples}\eqref{df example cospan shape}, which explains that this example is \emph{discretely flat}.} 
\end{enumerate}
\end{ex}

Thomason observed that Dwyer maps are stable under pushouts, as we now recall:

\begin{lem}[{\cite[Proposition 4.3]{ThomasonModelCat}}]
\label{pushoutDwyer}
Any pushout of a Dwyer map $I$ defines a Dwyer map $J$:
\[
\begin{tikzcd}
  \cA \arrow[d, "I"', hook]\arrow[r, "F"]\arrow[dr, phantom, "\ulcorner" very near end] &\cC \ar[d, "J", hook]\\
  \cB \arrow[r, "G" swap] &\cD.
\end{tikzcd}
\]
\end{lem}

Note for example, that \cref{pushoutDwyer} explains the Dwyer map of \cref{DwyerExamples}\eqref{ex:new-terminal-Dwyer}: if $\cA$ has a terminal object $t$, then the pushout
\[
\begin{tikzcd} \arrow[dr, phantom, "\ulcorner" very near end] \catone \arrow[d, hook, "0"'] \arrow[r, "t"] & \cA \arrow[d, hook] \\ \cattwo \arrow[r] & \cA^\triangleright
\end{tikzcd}
\]
defines the category $\cA^\triangleright$.

Our aim is to show that pushouts of categories involving at least one Dwyer map can also be regarded as pushouts of $(\infty,1)$-categories in the sense made precise by considering the nerve embedding from categories into quasi-categories:

\begin{thm} 
\label{DwyerPushout}
Let 
\[
\begin{tikzcd}
  \cA \arrow[d, "I"', hook]\arrow[r, "F"]\arrow[dr, phantom, "\ulcorner" very near end] &\cC \ar[d, "J", hook]\\
  \cB \arrow[r, "G" swap] &\cD
\end{tikzcd}
\]
be a pushout of categories, and assume $I$ to be a Dwyer map. Then the induced map of simplicial sets \[N\cB\aamalg{N\cA} N\cC \to N\cD\] is a weak categorical equivalence. 
 \end{thm}
 
By a weak categorical equivalence, we mean a weak equivalence in Joyal's model structure for quasi-categories \cite[\S 1]{JoyalTierney:QCSS}. \cref{DwyerPushout} is a refinement of a similar result of Thomason \cite[Proposition 4.3]{ThomasonModelCat} which proves that the same map is a weak homotopy equivalence.

While \cref{DwyerPushout} is the natural generalization of Thomason's result,  we prove it by considering instead the embedding of 1-categories as discrete simplicially enriched categories, using Bergner's model of $(\infty,1)$-categories. This tactic was suggested by an anonymous referee; for our original argument using the quasi-category model see \cite{arXiv-v2}. We show in \cref{prop we preservation} that a Dwyer map, when considered as a map between discrete simplicial categories, satisfies a certain ``flatness'' property with respect to the Bergner model structure. Since this discrete embedding of 1-categories into simplicial categories preserves pushouts, unlike the nerve embedding of 1-categories into quasi-categories, it is straightforward to prove that: 

\begin{thm}\label{model indep main theorem}
The inclusion $\Cat_1 \hookrightarrow \Cat_{(\infty,1)}$ of the $(\infty,1)$-category of $1$-categories into the $(\infty,1)$-category of $(\infty,1)$-categories preserves (homotopy) pushouts along Dwyer maps.
\end{thm}

When then deduce \cref{DwyerPushout} as a corollary of this result.

Though the two previous theorems refer to Dwyer maps, they also hold for the pseudo-Dwyer maps introduced by Cisinski \cite{CisinskiDwyer}, which are retracts of Dwyer maps.
In fact, we prove both \cref{DwyerPushout} and \cref{model indep main theorem} for more general classes of functors introduced in \cref{discretely flat} that include the Dwyer maps. The key property of a Dwyer map (or pseudo-Dwyer map) is that it is ``discretely flat'' as well as a faithful inclusion. For a functor to be \textbf{discretely flat} means that pushouts along it, considered as a functor of discrete simplicial categories, preserve Dwyer-Kan equivalences of simplicial categories.

In a companion paper, we give an application of \cref{DwyerPushout} to the theory of $(\infty,2)$-categories. There we prove:

\begin{thm}[{\cite[4.4.2]{HORR}}] The space of composites of any pasting diagram in any $(\infty,2)$-category is contractible.
\end{thm}

To prove this, we make use of Lurie's model structure of $(\infty,2)$-categories as categories enriched over quasi-categories \cite{LurieGoodwillie}. In this model, a \emph{pasting diagram} is a simplicially enriched functor out of the free simplicially enriched category defined by gluing together the objects, atomic 1-cells, and atomic 2-cells of a pasting scheme, while the composites of these cells belong to the homotopy coherent diagram indexed by the nerve of the free 2-category generated by the pasting scheme. 

This pair of $(\infty,2)$-categories has a common set of objects so the difference lies in their hom-spaces. The essential difference between the procedure of attaching an atomic 2-cell along the bottom of a pasting diagram or along the bottom of the free 2-category it generates is the difference between forming a pushout of hom-categories in the category of $(\infty,1)$-categories or in the category of 1-categories. Since one of the functors in the span that defines the pushout under consideration is a Dwyer map, \cref{DwyerPushout} proves that the resulting $(\infty,2)$-categories are equivalent. 

In \S\ref{sec:pushout}, we analyze 1-categorical pushouts of Dwyer maps. In \S\ref{sec:simp-pushout}, we extend these observations to pushouts of simplicial categories involving a Dwyer map between 1-categories as one leg of the span, axiomatize the classes of functors that are well-behaved with respect to simplicial pushouts, and prove \cref{model indep main theorem}. In \S\ref{sec:main-theorem}, we deduce \cref{DwyerPushout} as a corollary and consider a further special case \cref{AnodyneDwyer}, which observes that the canonical comparison between the pushout of nerves of categories and the nerve of the pushout is inner anodyne, provided that one of the functors in the span is a Dwyer map and the other is an injective on objects faithful functor.

\section{Dwyer pushouts}\label{sec:pushout}

We now establish some notation that we will freely reference in the remainder of this paper. By \cref{defn:Dwyer-map}, a Dwyer map $I \colon \cA \hookrightarrow \cB$ uniquely determines a functor 
$\chi \colon \cB \to \cattwo$ that classifies the sieve $\cA \coloneqq \chi^{-1}(0)$ and its complementary cosieve $\cV \coloneqq \chi^{-1}(1)$
\begin{equation*}
\begin{tikzcd} \cV \arrow[r, hook] \arrow[d] \arrow[dr, phantom, "\lrcorner" very near start] & \cB \arrow[d, "\chi"] & \arrow[dl, phantom, "\llcorner" very near start] \cA \arrow[l, hook'] \arrow[d] \\ \catone \arrow[r, hook, "1"'] & \cattwo & \catone \arrow[l, hook', "0"]
\end{tikzcd}
\end{equation*}
 as well as a right adjoint left inverse adjunction $(I \dashv R,\varepsilon \colon IR \Rightarrow \id_\cW)$ associated to the inclusion of $\cA$ into the minimal cosieve $\cA \subset\cW\subset\cB$. This data may be summarized by the diagram
\begin{equation}\label{eq:DwyerData}
\begin{tikzcd}
& & \varnothing  \arrow[dl] \arrow[dr] \arrow[dd, phantom, "\rotatebox{135}{$\ulcorner$}" very near start]  \\ & \cU \arrow[dl, hook'] \arrow[dr, hook] \arrow[dd, phantom, "\rotatebox{135}{$\ulcorner$}" very near start]& 
  & \cA \arrow[dl, hook'] \arrow[d]  \arrow[ddl, phantom, "\llcorner" very near start] \\
\cV \arrow[d] \arrow[dr, hook] \arrow[ddr, phantom, "\lrcorner" very near start]  & & \cW \arrow[dl, hook'] \arrow[ur, dashed, "R", bend left,  "\rotatebox{45}{$\top$}"' outer sep=-2pt]& \catone \arrow[ddll, "0", hook'] \\
 \catone \arrow[dr, hook, "1"'] & \cB \arrow[d] & ~ &  \\& \cattwo
\end{tikzcd}
\end{equation}
 in which $\cU \coloneqq\cW\cap\cV\cong \cW\backslash\cA $. Consider the pushout of a Dwyer map along an arbitrary functor $F \colon \cA \to \cC$. 
\begin{equation*}
\begin{tikzcd}[sep=small] & [+10pt] ~& \cA \arrow[dl, hook', "I"'] \arrow[dr, phantom, "\ulcorner" very near end] \arrow[rr, "F"] \arrow[dd] & & \cC \arrow[dd] \arrow[dl, "J", hook'] \\ \cV \arrow[r, hook] \arrow[dd] &  \cB  \arrow[dd, "\chi"'] \arrow[rr, "G"' near end, crossing over] & & \cD  \\ 
& & \catone \arrow[rr, equals] \arrow[dl, hook', "0"'] & & \catone \arrow[dl, hook', "0"] \\ \catone \arrow[r, hook, "1"] &  \cattwo \arrow[rr, equals] & & \cattwo \arrow[from=uu, "\pi" near start, dashed, crossing over] \end{tikzcd}
\end{equation*}
The induced functor $\pi \colon \cD \to \cattwo$ partitions the objects of $\cD$ into the two fibers $\ob(\pi^{-1}(0)) \cong \ob\cC$ and $\ob(\pi^{-1}(1))\cong \ob\cV$ and prohibits any morphisms from the latter to the former.

The right adjoint left inverse adjunction $(I \dashv R,\varepsilon \colon IR \Rightarrow \id_\cW)$ associated to the inclusion of $\cA$ into the minimal cosieve $\cA \subset\cW\subset\cB$  pushes out to define a right adjoint left inverse $(J \dashv S, \nu \colon JS \Rightarrow \id_{\cY})$ to the inclusion of $\cC$ into the minimal cosieve $\cC \subset\cY\subset\cD$.
\[
\begin{tikzcd}
 \cA \arrow[r, "F"] \arrow[d, hook, bend right, "I"'] \arrow[d, phantom, "\dashv"] \arrow[dr, phantom, "\ulcorner" very near end] & \cC \arrow[d, hook, bend right, "J"'] \arrow[d, phantom, "\dashv"] &  \cA \arrow[r, "F"] \arrow[d, hook, "I"'] \arrow[dr, phantom, "\ulcorner" very near end] & \cC \arrow[d, hook, "J"] \arrow[ddr, bend left, equals] & &\cA \arrow[r, "F"] \arrow[d, hook, "I"'] \arrow[dr, phantom, "\ulcorner" very near end] & \cC \arrow[d, hook, "J"] \arrow[dr,  "\Delta"] \\
 \cW \arrow[r, "G"'] \arrow[d, hook] \arrow[dr, phantom, "\ulcorner" very near end] \arrow[u, bend right, "R"'] & \cY \arrow[d, hook] \arrow[u, bend right, "S"'] & \cW \arrow[r, "G"'] \arrow[dr, "R"'] & \cY \arrow[dr, dashed, "S"] & & \cW \arrow[dr, "\varepsilon"'] \arrow[r, "G"'] & \cY \arrow[dr, dashed, "\nu"] & \cC^\cattwo \arrow[d, "J^\cattwo"] \\ \cB \arrow[r, "G"'] & \cD & & \cA \arrow[r, "F"'] & \cC & & \cW^\cattwo \arrow[r, "G^\cattwo"'] & \cY^\cattwo
\end{tikzcd}
\]

These observations explain the closure of Dwyer maps under pushout and furthermore can be used to explicitly describe the structure of the category $\cD$ defined by the pushout of a Dwyer map, as proven in \cite[Proof of Lemma 2.5]{BMOOPY}; cf.\ also \cite[Construction 1.2]{SchwedeOrbispaces} and \cite[\textsection7.1]{AraMaltsiniotisVers}.

\begin{prop}
\label{prop:PushoutDwyerHom}
The objects in the pushout category $\cD$ are given by
\[
\begin{tikzcd} \ob\cC\amalg\ob\cV \arrow[r, iso] & \ob\cD
\end{tikzcd}
\]
while the  hom sets are given by
\[
\begin{tikzcd}[row sep=tiny]
\cC(c,c') \cong \cD(c,c')  & \cV(v,v') \cong \cB(v,v') \cong \cD(v,v')  &  \cC(c,Su) \arrow[r, "\nu_{u} \circ (-)", iso'] & \cD(c,u) \\ & & f \rar[mapsto] &\hat f
\end{tikzcd}
\]
for all $c,c' \in \cC$, $v,v' \in \cV$, and $u \in \cU$, and are empty otherwise. Functoriality of the inclusions $J$ and $G$ defines the composition on the image of $\cC$ and $\cV$. For objects $c,c' \in \cC$ and $u,u' \in \cU$, the composition map
\[
\begin{tikzcd} \cD(u,u') \times \cD(c, u) \times \cD(c',c) \arrow[r, "\circ"] & \cD(c', u')  \\
\arrow[u, iso]
\cD(u,u') \times \cC(c,Su) \times \cC(c',c) \arrow[d, "S \times \id"'] \\ \cC(Su,Su') \times \cC(c, Su) \times \cC(c',c)  \arrow[r, "\circ"]  & \cC(c',Su')  \arrow[uu, iso']
\end{tikzcd}
\]
is the unique map making the diagram commute.\footnote{Note if $u \in \cU$ and $v \in \cV\backslash\cU$, then $\cB(u,v)=\varnothing$.}
\end{prop}
To summarize, $J$ and $G$ define fully-faithful inclusions
\begin{equation*}
\begin{tikzcd} \cV \arrow[d] \arrow[r, hook] \arrow[dr, phantom, "\lrcorner" very near start] & \cD \arrow[d, "\pi"] & \cC \arrow[l, hook'] \arrow[d] \arrow[dl, phantom, "\llcorner" very near start] \\ \catone \arrow[r, hook, "1"] & \cattwo & \catone \arrow[l, hook', "0"']
\end{tikzcd}
\end{equation*}
that are jointly surjective on objects. In particular, we may identify $\cV$ with the complementary cosieve of $\cC$ in $\cD$. 

\section{Pushouts in simplicial categories}\label{sec:simp-pushout}

By \emph{simplicial category} we always mean \emph{simplicially enriched category} as opposed to a simplicial object in $\Cat$. There is a fully faithful inclusion $\sCat \hookrightarrow \Cat^{\DDelta^\op}$, identifying a simplicial category with an identity-on-objects simplicial object. 

The following model structure on $\sCat$ is due to Bergner \cite{Bergner:MS}, though Lurie \cite[A.3.2.4, A.3.2.25]{HTT} observed that the Bergner model structure is left proper and combinatorial.

\begin{defn}[Bergner model structure]
The category $\sCat$ of simplicially enriched categories admits a proper, combinatorial model structure in which:
\begin{itemize}
    \item A map $f \colon \cC\to \cD$ is a \textbf{weak equivalence} just when:
\begin{enumerate}[label=(W\arabic*), ref=W\arabic*]
    \item For each pair of objects $x,y$, the map $\cC(x,y) \to \cD(fx,fy)$ is a weak homotopy equivalence of simplicial sets, and \label{DK local equiv} 
    \item the functor $\pi_0 f \colon \pi_0 \cC \to \pi_0 \cD$ is essentially surjective. \label{DK global equiv}
\end{enumerate}
\item A map $f \colon \cC\to \cD$ is a \textbf{fibration} just when:
\begin{enumerate}[label=(F\arabic*), ref=F\arabic*]
    \item For each pair of objects $x,y$, the map $\cC(x,y) \to \cD(fx,fy)$ is a Kan fibration, and \label{DK local fib}
    \item the functor $\pi_0 f \colon \pi_0 \cC \to \pi_0 \cD$ is an isofibration.
    \label{DK global fib}
\end{enumerate}
\end{itemize}
If $\cC$ is a simplicial category, then $\pi_0 \cC$ is the ordinary category obtained by taking path components of each hom-simplicial set.
We call simplicial functors satisfying \eqref{DK local equiv} \emph{fully faithful} (meaning of course in the homotopical sense), and functors satisfying \eqref{DK global equiv} \emph{essentially surjective}.
\end{defn}

The constant diagram functor $\Cat \to \Cat^{\DDelta^\op}$, given by precomposition with $\DDelta^\op \to \catone$, factors through the full subcategory inclusion $\sCat \hookrightarrow \Cat^{\DDelta^\op}$. Write $\cstar \colon \Cat \to \sCat$ for the induced full inclusion, which identifies categories as those simplicial categories with discrete hom simplicial sets. 

\begin{lem}
\label{ConstColims}
The functor $\cstar \colon \Cat \to \sCat$ preserves limits and colimits.
\end{lem}
\begin{proof}
The precomposition functor $\Cat \to \Cat^{\DDelta^\op}$ preserves all limits and colimits, while the full inclusion $\sCat \to \Cat^{\DDelta^\op}$ reflects them. The conclusion follows.
\end{proof}

Our key technical result is the following proposition, which observes that when $I$ is a Dwyer map, the functor $\cstar(I)$ of simplicial categories is a \emph{flat map} in the terminology of \cite[B.9]{HillHopkinsRavenel:ONEKIO} or an \emph{$h$-cofibration} in the terminology of \cite[1.1]{BataninBerger:HTAPM} relative to the Bergner model structure.

\begin{prop}
\label{prop we preservation}
If $I \colon \cA \hookrightarrow \cB$ is a Dwyer map, then \[\cstar\cB \aamalg{\cstar\cA} (-) \colon \slicel{\cstar\cA}{\sCat} \to \slicel{\cstar\cB}{\sCat} \] preserves weak equivalences.
\end{prop}

\begin{proof}
Consider a composable pair of simplicial functors $\cstar\cA \xrightarrow{F} \cC' \xrightarrow{M} \cC$ and form the following pushouts:
\[ \begin{tikzcd}
\cstar\cA \arrow[dr, phantom, "\ulcorner" very near end] \dar \rar{F} & \cC' \rar{M} \dar \arrow[dr, phantom, "\ulcorner" very near end] & \cC \dar \\ 
\cstar\cB \rar & \cD' \rar{H} & \cD
\end{tikzcd} \]
When $M$ is a weak equivalence in the Bergner model structure,  we wish to show that the induced map $H \colon \cD' \to \cD$ between the pushouts is as well.

As in \cref{prop:PushoutDwyerHom}, we identify $\ob\cD'$ with $\ob\cC' \amalg \ob\cV$ and similarly $\ob\cD = \ob\cC \amalg \ob\cV$.
We regard the simplicial categories $\cD', \cD$ as simplicial objects $\cD'_\bullet, \cD_\bullet$ via the inclusion $\sCat \hookrightarrow \Cat^{\DDelta^\op}$.
For each $n$, we have $\cD_n = \cB \amalg_{\cA} \cC_n$ and similarly for $\cD_n'$. We have already computed the hom sets of these categories in \cref{prop:PushoutDwyerHom}, and the descriptions there are functorial in the $\cC$-variable.
Thus we have 
\[ \begin{tikzcd}
\cC'(c,c') \rar{\simeq} \dar{\cong} & \cC(Mc, Mc') \dar{\cong} & \cC'(c,FRu) \rar{\simeq} \dar{\cong} & \cC(Mc, MFRu) \dar{\cong} \\
\cD'(c,c') \rar & \cD(Hc,Hc') & \cD'(c,u) \rar & \cD(Hc,Hu)
\end{tikzcd} \]
for $c,c' \in \cC'$ and $u\in \cU$.
Meanwhile, for $v,v' \in \cV$ we have that both $\cD(v,v')$ and $\cD'(v,v')$ are isomorphic to the discrete simplicial set $\cV(v,v')$.
Finally, the hom simplicial sets $\cD'(c,v')$, $\cD(Hc,v')$, $\cD(v,c)$, and $\cD(v,Hc)$ are all empty for $c\in \cC'$, $v\in \cV$, $v' \in \cV \backslash \cU$. 
Thus $H$ is fully faithful.

For essential surjectivity of $H$, notice that we have a commutative square of functors
\[ \begin{tikzcd}
\cC' \amalg \cstar\cV \rar{M \amalg \id} \dar & \cC \amalg \cstar\cV \dar \\
\cD' \rar{H} & \cD
\end{tikzcd} \]
where the vertical maps are bijective on objects and the top map is essentially surjective.
It follows that $H$ is essentially surjective as well.
\end{proof}

Our main results hold not just for Dwyer maps but for arbitrary functors between 1-categories that satisfy the property established in \cref{prop we preservation} plus some injectivity conditions. The following terminology highlights the required properties. 

\begin{defn}\label{discretely flat}
A functor $I \colon \cA \to \cB$ between 1-categories is \textbf{discretely flat} if the simplicial functor $\cstar(I)$ is flat, i.e., if \[\cstar\cB \aamalg{\cstar\cA} (-) \colon \slicel{\cstar\cA}{\sCat} \to \slicel{\cstar\cB}{\sCat} \] preserves Bergner weak equivalences. If, in addition, $I$ is injective on objects, we call it a \textbf{discretely flat cofibration}, and if it is both injective on objects and faithful, we call it a \textbf{discretely flat inclusion}.
\end{defn}

Dwyer maps are discretely flat inclusions, but such functors aren't the only examples.

\begin{ex}\label{discretely flat examples}
$\quad$
\begin{enumerate}[label=(\roman*),ref=\roman*]
\item Since the passage to opposite categories commutes with all of the structures involved in \cref{discretely flat}, {\bf co-Dwyer maps}, whose opposites are Dwyer maps, are also discretely flat inclusions.
\item As flat maps are closed under retracts --- see \cite[B.11]{HillHopkinsRavenel:ONEKIO} or \cite[1.3]{BataninBerger:HTAPM} --- Cisinski's {\bf pseudo-Dwyer maps} \cite{CisinskiDwyer} are also discretely flat inclusions.
\item The inclusion of $\catone \amalg \catone$ into the cospan category $(\catone \amalg \catone)^\triangleright$ is a discretely flat inclusion. Indeed, the hom simplicial sets of $(\catone \amalg \catone)^{\triangleright} \amalg_{(\catone \amalg \catone)} \cC$ are readily computed by hand in terms of those for $\cC$, and a variation of the proof of \cref{prop we preservation} gives the result.\label{df example cospan shape}
\end{enumerate}
\end{ex}

Note not all functors of the form $\cA \to \cA^{\triangleright}$ are discrete flat inclusions. In light of \cref{model indep main theorem}, the left-hand map of \eqref{eq:poset-non-example} gives a counter-example. An interesting problem is to characterize the class of discretely flat inclusions. 

The fact that Dwyer pushouts are homotopy pushouts now follows from a general fact:  in a left proper model category, a pushout where one leg is a flat map is automatically a homotopy pushout (see \cite[\S1.5]{BataninBerger:HTAPM}).

\begin{prop}
\label{simp cat theorem}
Suppose $I \colon \cA \to \cB$ is discretely flat. Then for any functor $F \colon \cstar \cA \to \cC$ of simplicial categories, the pushout $\cstar \cB \aamalg{\cstar \cA} \cC$ is a homotopy pushout.
\end{prop}
\begin{proof}
To form the homotopy pushout of a span in a model category, one replaces it by a cofibrant span as below and then takes the ordinary pushout (\cite[10.4]{DwyerSpalinski}):
\[ \begin{tikzcd}
& \varnothing \dar[tail] \\
\cY  \dar[we] & \mathcal{X} \lar[tail] \rar[tail]  \dar[we] & \mathcal{Z} \dar[we] \\
\cstar \cB  & \cstar \cA \lar["\cstar(I)"'] \rar[tail] & \cC 
\end{tikzcd} \]
Thus, we must show that the induced map
\[
  \cY \aamalg{\mathcal{X}} \mathcal{Z} \to \cstar \cB \aamalg{\cstar \cA} \cC
\]
is a weak equivalence of simplicial categories.
Since $\cstar(I)$ is flat by assumption and $\sCat$ is left proper, the above map is a weak equivalence by \cite[B.12]{HillHopkinsRavenel:ONEKIO} 
\end{proof}

Our model-independent statement, that Dwyer pushouts are $(\infty,1)$-categorical, holds generally for discretely flat cofibrations. 

{
\renewcommand{\thethm}{\ref{model indep main theorem}}
\begin{thm}
The inclusion $\Cat_1 \hookrightarrow \Cat_{(\infty,1)}$ of the $(\infty,1)$-category of $1$-categories into the $(\infty,1)$-category of $(\infty,1)$-categories preserves (homotopy) pushouts along Dwyer maps.
\end{thm}
\begin{proof}
The inclusion $\Cat_1 \hookrightarrow \Cat_{(\infty,1)}$ can be modeled at the point-set level by the right Quillen functor $\cstar \colon \Cat \to \sCat$.  By hypothesis, a discretely flat cofibration $I\colon \cA \to \cB$ is injective on objects, hence a cofibration in the canonical model structure on $\Cat$. 
This model structure is left proper, so ordinary pushouts along such maps are homotopy pushouts by  \cite[13.5.4]{Hirschhorn:MCL}. Since the functor $\cstar \colon \Cat \to \sCat$ preserves strict pushouts, \Cref{simp cat theorem} shows that $\cstar$ preserves homotopy pushouts along discretely flat cofibrations, so the conclusion follows.
\end{proof}
\addtocounter{thm}{-1}
}

\section{Pushouts in simplicial sets}\label{sec:main-theorem}

In this section, we describe the implications of \cref{model indep main theorem} for the Joyal model structure on simplicial sets, proving the results needed in \cite{HORR}. We utilize the following commutative triangle of right Quillen functors
\[ \begin{tikzcd}[column sep=small]
& \Cat \ar[dl,"N"'] \ar[dr,"\cstar"] \\
\sSet & & \sCat \ar[ll,"\hcnerve", "\simeq"']
\end{tikzcd} \]
where $\hcnerve$ is the homotopy coherent nerve, a right Quillen equivalence between the Bergner and Joyal model structures.
This diagram commutes up to natural isomorphism since for any 1-category $\cA$
\[
  (\hcnerve \cstar \cA)_n \coloneqq \hom(\mathfrak{C}[n], \cstar \cA) \cong \hom([n], \cA) \eqqcolon (N\cA)_n 
\]
which holds because the hom simplicial sets of $\cstar \cA$ are discrete.

We first explain how to deduce \cref{DwyerPushout} from \cref{simp cat theorem}. In fact, we use the terminology of \cref{discretely flat} to prove a more general version:

{
\renewcommand{\thethm}{\ref{DwyerPushout}}
\begin{thm}
Let 
\[
\begin{tikzcd}
  \cA \arrow[d, "I"', hook]\arrow[r, "F"]\arrow[dr, phantom, "\ulcorner" very near end] &\cC \ar[d, "J", hook]\\
  \cB \arrow[r, "G" swap] &\cD
\end{tikzcd}
\]
be a pushout of categories, and assume $I$ is a discretely flat inclusion. Then the induced map of simplicial sets \[N\cB\aamalg{N\cA} N\cC \to N\cD\] is a weak categorical equivalence. 
 \end{thm}
 \addtocounter{thm}{-1}
}

\begin{proof}
We organize the proof into the following commutative square of simplicial sets:
\[ \begin{tikzcd}
\hcnerve \cstar \cB \aamalg{\hcnerve \cstar \cA} \hcnerve \cstar \cC \rar{\cong} \dar & N\cB \aamalg{N\cA} N\cC \ar[dd] \\
\hcnerve \left( \cstar \cB \aamalg{\cstar \cA} \cstar \cC \right) \dar["\cong"] \\
\hcnerve \cstar \left( \cB \aamalg{\cA}  \cC \right) \rar{\cong} & N(\cB \aamalg{\cA}  \cC)
\end{tikzcd} \]
The top and bottom isomorphisms are instances of the natural isomorphism $N \cong \hcnerve \cstar$. The vertical maps are the canonical comparison maps induced by the universal property of the pushouts. The bottom left map is an isomorphism since $\cstar$ preserves pushouts (\cref{ConstColims}).
It remains to show that the upper left map
\[
  \hcnerve \cstar \cB \aamalg{\hcnerve \cstar \cA} \hcnerve \cstar \cC \to \hcnerve \left( \cstar \cB \aamalg{\cstar \cA} \cstar \cC \right)
\]
is a weak categorical equivalence, at which point the right map will be a weak categorical equivalence by two-of-three.

Notice that the objects in the top row are homotopy pushouts, since $NI \colon N\cA \to N\cB$ is a cofibration in $\sSet$ by the hypothesis that $I$ is faithful and injective on objects. Let $\mathbf{R}\hcnerve$ be the right derived functor of $\hcnerve$.
Since discrete simplicial categories are fibrant in the Bergner model structure, the map of simplicial sets above represents the canonical map
\[
  (\mathbf{R}\hcnerve) \cstar \cB \overset{h}{\aamalg{(\mathbf{R}\hcnerve) \cstar \cA}} (\mathbf{R}\hcnerve) \cstar \cC \to (\mathbf{R}\hcnerve) \left( \cstar \cB \overset{h}{\aamalg{\cstar \cA}} \cstar \cC \right).
\]
This is an equivalence since $\mathbf{R}\hcnerve$ is an equivalence of $(\infty,1)$-categories, hence preserves homotopy pushouts.
We conclude that $N\cB \aamalg{N\cA} N\cC \to N(\cB \aamalg{\cA} \cC)$ is a weak categorical equivalence.
\end{proof}

In the case where $F \colon \cA \to \cC$ is also injective on objects and faithful, as occurs frequently in \cite{HORR}, we are able to strengthen our conclusion and prove that the canonical comparison map is inner anodyne. 

\begin{cor} 
\label{AnodyneDwyer}
Let 
\[
\begin{tikzcd}
  \cA \arrow[d, "I"', hook]\arrow[r, "F"]\arrow[dr, phantom, "\ulcorner" very near end] &\cC \ar[d, "J", hook]\\
  \cB \arrow[r, "G" swap] &\cD
\end{tikzcd}
\]
be a pushout of categories, in which $I$ is a discretely flat inclusion and $F$ is faithful and injective on objects. Then the induced inclusion of simplicial sets \[N\cB\aamalg{N\cA} N\cC \hookrightarrow N\cD\] is inner anodyne.
 \end{cor}
 \begin{proof}
 Observe in this case that the canonical map $j \colon N\cB\amalg_{N\cA}N\cC \hookrightarrow N\cD$ is an inclusion and thus, by \cref{DwyerPushout}, an acyclic cofibration in the Joyal model structure. This acyclic cofibration is also bijective on 0-simplices and has codomain a quasi-category. By \cite[2.19]{Stevenson:MSCB} or \cite[5.7]{Stevenson:NJMS} it follows that  $j$ is inner anodyne.\footnote{See \cite{Campbell:CQCT} for related discussion and an example of an acyclic cofibration that is bijective on 0-simplices but whose codomain is not a quasi-category that is not inner anodyne.}
\end{proof}

\bibliographystyle{alpha}
\bibliography{refs_inf-2}

\end{document}